\documentclass[preprint,9pt]{elsarticle}
\usepackage{amssymb}
\usepackage{amsmath}
\usepackage{amsthm}
\usepackage{enumerate}

\newtheorem{thrm}{Theorem}[section]

\newtheorem{exam}[thrm]{Example}
\newtheorem{cor}[thrm]{Corollary}

\theoremstyle{definition}

\newtheorem{remark}[thrm]{Remark}

\journal{...}
\input{tcilatex}

\begin{document}

\begin{frontmatter}


\cortext[cor1]{Corresponding author (+903562521616-3087)}

\title{Determinantal and Permanental Representation of Generalized Fibonacci Polynomials}


\author[rvt]{Adem Sahin}
\ead{adem.sahin@gop.edu.tr} {\author[rvt]{Kenan Kaygisiz
\corref{cor1}}} \ead{kenan.kaygisiz@gop.edu.tr}
\address[rvt]{Department of Mathematics, Faculty of Arts and Sciences,
Gaziosmanpa\c{s}a University, 60250 Tokat, Turkey}

\begin{abstract}
In this paper, we give some determinantal and permanental
representations of Generalized Fibonacci Polynomials by using
various Hessenberg matrices. These results are general form of determinantal and permanental representations of $k$ sequences of the generalized order-$k$ Fibonacci and Pell numbers.
\end{abstract}

\end{frontmatter}



\section{Introduction}

Fibonacci numbers, Pell numbers and their generalizations have been studying
for a long time. One of these generalizations was given by Miles in 1960.

Miles [8] defined generalized order-$k$ Fibonacci numbers(GO$k$F) as,%
\begin{equation}
f_{k,n}=\sum\limits_{j=1}^{k}f_{k,n-j}\
\end{equation}%
for $n>k\geq 2$, with boundary conditions: $f_{k,1}=f_{k,2}=f_{k,3}=\cdots
=f_{k,k-2}=0$ and $f_{k,k-1}=f_{k,k}=1.$\newline

Er [2] defined $k$ sequences of generalized order-$k$ Fibonacci numbers ($k$%
SO$k$F) as; for $n>0,$ $1\leq i\leq k$%
\begin{equation}
f_{k,n}^{\text{ }i}=\sum\limits_{j=1}^{k}c_{j}f_{k,n-j}^{\text{ }i}\ \
\end{equation}%
with boundary conditions for $1-k\leq n\leq 0,$

\begin{equation*}
f_{k,n}^{\text{ }i}=\left\{
\begin{array}{l}
1\text{ \ \ \ \ \ if \ }i=1-n, \\
0\text{ \ \ \ \ \ otherwise,}%
\end{array}%
\right.
\end{equation*}%
where $c_{j}$ $(1\leq j\leq k)$ are constant coefficients, $f_{k,n}^{\text{ }%
i}$ is the $n$-th term of $i$-th sequence of order-$k$ generalization. For $%
c_{j}=1$, $k$-th sequence of this generalization involves the Miles
generalization(1) for $i=k,$ i.e.%
\begin{equation*}
f_{k,n}^{k}=f_{k,k+n-2}.
\end{equation*}

Kili\c{c} and Ta\c{c}c\i \lbrack 4] defined $k$ sequences of generalized
order-$k$ Pell numbers ($k$SO$k$P) as; for $n>0,$ $1\leq i\leq k$%
\begin{equation}
p_{k,n}^{\text{ }i}=2p_{k,n-1}^{\text{ }i}+p_{k,n-2}^{\text{ }i}+\cdots +\
p_{k,n-k}^{\text{ }i}\
\end{equation}%
with initial conditions for $1-k\leq n\leq 0,$

\begin{equation*}
p_{k,n}^{\text{ }i}=\left\{
\begin{array}{l}
1\text{ \ \ \ \ \ if \ }i=1-n, \\
0\text{ \ \ \ \ \ otherwise,}%
\end{array}%
\right.
\end{equation*}%
where $p_{k,n}^{\text{ }i}$ is the $n$-th term of $i$-th sequence of order $%
k $ generalization.

Kaygisiz and \c{S}ahin [3] defined $k$ sequences of the generalized order-$k$
Van der Laan numbers($k$SO$k$V) as; for $n>0,$ $1\leq i\leq k$%
\begin{equation}
v_{k,n}^{i}=\sum\limits_{j=2}^{k}v_{k,n-j}^{i}\
\end{equation}%
with initial conditions for $1-k\leq n\leq 0,$

\begin{equation*}
v_{k,n}^{i}=\left\{
\begin{array}{lc}
1\ \ \ \ \ \ \ \ \text{ \ if }i-n=k, &  \\
0\ \text{ \ \ \ \ \ \ \ \ otherwise} &
\end{array}%
\right.
\end{equation*}%
for $1-k\leq n\leq 0,$ where $v_{k,n}^{i}$ is the $n$-th term of $i$-th
sequence.

\bigskip

MacHenry [5] defined generalized Fibonacci polynomials $(F_{k,n}(t))$ where $%
t_{i}$ $(1\leq i\leq k)$ are constant coefficients of the core polynomial%
\begin{equation*}
P(x;t_{1},t_{2},\ldots ,t_{k})=x^{k}-t_{1}x^{k-1}-\cdots -t_{k},
\end{equation*}%
which is denoted by the vector
\begin{equation*}
t=(t_{1},t_{2},\ldots ,t_{k}).
\end{equation*}%
$F_{k,n}(t)$ is defined inductively by
\begin{eqnarray}
F_{k,n}(t) &=&0,\text{ }n<1 \\
F_{k,1}(t) &=&1  \notag \\
F_{k,2}(t) &=&t_{1}  \notag \\
F_{k,n+1}(t) &=&t_{1}F_{k,n}(t)+\cdots +t_{k}F_{k,n-k+1}(t).  \notag
\end{eqnarray}%
For example the generalized Fibonacci polynomials for $k=4$ and $k=5$ are;%
\begin{equation*}
1,t_{1},t_{2}+t_{1}^{2},t_{3}+2t_{1}t_{2}+t_{1}^{3},t_{4}+2t_{1}t_{3}+t_{2}^{2}+t_{1}^{4}+3t_{1}^{2}t_{2},...
\end{equation*}

\bigskip\ and%
\begin{eqnarray*}
&&1,\text{ }t_{1},\text{ }t_{2}+t_{1}^{2},\text{ }%
t_{3}+2t_{1}t_{2}+t_{1}^{3},\text{ }%
t_{4}+2t_{1}t_{3}+t_{2}^{2}+t_{1}^{4}+3t_{1}^{2}t_{2},\text{ } \\
&&t_{1}^{5}+4t_{1}^{3}t_{2}+3t_{1}^{2}t_{3}+3t_{1}t_{2}^{2}+2t_{1}t_{4}+2t_{2}t_{3}+t_{5},
\\
&&2t_{1}t_{5}+2t_{2}t_{4}+6t_{1}t_{2}t_{3}+t_{2}^{3}+t_{3}^{2}+t_{1}^{6}+%
\allowbreak
3t_{1}^{2}t_{4}+4t_{1}^{3}t_{3}+5t_{1}^{4}t_{2}+6t_{1}^{2}t_{2}^{2},...
\end{eqnarray*}%
respectively.

\bigskip MacHenry studied on these polinomials and obtain very useful
properties of these polynomials in [6,7].

\begin{remark}
$f_{k,n},$ $f_{k,n}^{\text{ }i},$ $p_{k,n}^{\text{ }i}$, $v_{k,n}^{i}$ and $%
F_{k,n}(t)$ are GO$k$F (1), $k$SO$k$F (2), $k$SO$k$P (3), $k$SO$k$V (4) and
generalized Fibonacci polynomials (5) respectively, then \newline
$i)$ substituting $c_{j}=t_{i}$ in (2) and generalized Fibonacci
polynomials, for $1\leq i, j\leq k$ we obtain%
\begin{equation*}
F_{k,n}(t)=f_{k,n-1}^{\text{ }1},
\end{equation*}%
$ii)$ substituting $t_{1}=2$ and $t_{i}=1$ for $2\leq i\leq k$ in
generalized Fibonacci polynomials, we obtain%
\begin{equation*}
F_{k,n}(t)=p_{k,n}^{\text{ }k},
\end{equation*}%
$iii)$ substituting $t_{1}=0$ and $t_{i}=1$ for $2\leq i\leq k$ in
generalized Fibonacci polynomials, we obtain
\begin{equation*}
F_{k,n}(t)=v_{k,n}^{k},
\end{equation*}%
$iv)$ substituting $t_{i}=1$ in generalized Fibonacci polynomials, we obtain
\begin{equation*}
F_{k,n}(t)=f_{k,k+n-2}.
\end{equation*}
\end{remark}

Minc [9] found Hessenberg matrix whose permanent are subscripted generalized
order-$k$ Fibonacci numbers. Ocal [10] gave various Hessenberg matrices
whose determinants and permanents are subscripted generalized order-$k$
Fibonacci numbers.

In this paper we derive determinantal and permanental representation of
generalized Fibonacci polynomials using various Hessenberg matrices.

\section{The determinantal representations}

\bigskip

An $n\times n$ matrix $A_{n}=(a_{ij})$ is called lower Hessenberg matrix if $%
a_{ij}=0$ when $j-i>1$ i.e.,%
\begin{equation*}
A_{n}=\left[
\begin{array}{ccccc}
a_{11} & a_{12} & 0 & \cdots & 0 \\
a_{21} & a_{22} & a_{23} & \cdots & 0 \\
a_{31} & a_{32} & a_{33} & \cdots & 0 \\
\vdots & \vdots & \vdots &  & \vdots \\
a_{n-1,1} & a_{n-1,2} & a_{n-1,3} & \cdots & a_{n-1,n} \\
a_{n,1} & a_{n,2} & a_{n,3} & \cdots & a_{n,n}%
\end{array}%
\right]
\end{equation*}

\begin{thrm}
$\bigskip \lbrack 1]$ $A_{n}$ be the $n\times n$ lower Hessenberg matrix for
all $n\geq 1$ and define $\det (A_{0})=1,$ then,%
\begin{equation*}
\det (A_{1})=a_{11}
\end{equation*}%
and for $n\geq 2$%
\begin{equation}
\det (A_{n})=a_{n,n}\det
(A_{n-1})+\sum\limits_{r=1}^{n-1}((-1)^{n-r}a_{n,r}\prod%
\limits_{j=r}^{n-1}a_{j,j+1}\det (A_{r-1})).
\end{equation}
\end{thrm}

\bigskip

\begin{thrm}
Let $k\geq 2$ be an integer$,$ $F_{k,n}(t)$ be the generalized Fibonacci
Polynomial (5) and $Q_{k,n}=(q_{rs})$ $n\times n$ Hessenberg matrix$,$ where%
\begin{equation*}
q_{rs}=\left\{
\begin{array}{c}
i^{\left\vert r-s\right\vert }.\frac{t_{r-s+1}}{t_{2}^{(r-s)}}\text{ \ \ \ \
\ \ \ \ \ \ \ \ \ \ \ \ \ \ if \ }-1\leq r-s<k\text{ },\text{\ \ \ \ \ \ \ \
\ \ \ \ \ \ \ \ \ \ \ \ \ \ \ \ \ \ } \\
0\text{ \ \ \ \ \ \ \ \ \ \ \ \ \ \ \ \ \ \ \ \ \ \ otherwise\ \ \ \ \ \ \ \
\ \ \ \ \ \ \ \ \ \ \ \ \ \ \ \ \ \ \ \ \ \ \ \ \ \ }%
\end{array}%
\right.
\end{equation*}%
i.e.,%
\begin{equation}
Q_{k,n}=\left[
\begin{array}{cccccc}
t_{1} & it_{2} & 0 & 0 & \cdots  & 0 \\
i & t_{1} & it_{2} & 0 & \cdots  & 0 \\
i^{2}\frac{t_{3}}{t_{2}^{2}} & i & t_{1} & it_{2} & \cdots  & 0 \\
\vdots  & \vdots  & \vdots  & \vdots  &  & \vdots  \\
i^{k-1}\frac{t_{k}}{t_{2}^{k-1}} & i^{k-2}\frac{t_{k-1}}{t_{2}^{k-2}} &
i^{k-3}\frac{t_{k-2}}{t_{2}^{k-3}} & i^{k-4}\frac{t_{k-3}}{t_{2}^{k-4}} &
\cdots  & 0 \\
0 & i^{k-1}\frac{t_{k}}{t_{2}^{k-1}} & i^{k-2}\frac{t_{k-1}}{t_{2}^{k-2}} &
i^{k-3}\frac{t_{k-2}}{t_{2}^{k-3}} & \cdots  & 0 \\
& \vdots  & \vdots  & \vdots  & \ddots  &  \\
0 & 0 & 0 & \cdots  & \cdots  & t_{1}%
\end{array}%
\right]   \label{kuka}
\end{equation}%
then%
\begin{equation*}
\det (Q_{k,n})=F_{k,n+1}(t)
\end{equation*}%
where $t_{0}=1$ and $i=\sqrt{-1}.$
\end{thrm}

\bigskip

\begin{proof}
\bigskip To prove $\det (Q_{k,m})=F_{k,m+1}(t)$, we use the mathematical induction on $m$. The result is true for $%
m=1$ by hypothesis.

Using Theorem (2.1) we have
\begin{eqnarray*}
\det (Q_{k,m+1}) &=&q_{m+1,m+1}\det
(Q_{k,m})+\sum\limits_{r=1}^{m}\left(
(-1)^{m+1-r}q_{m+1,r}\prod\limits_{j=r}^{m}q_{j,j+1}\det
(Q_{k,r-1})\right)
\\
&=&t_{1}\det (Q_{k,m})+\sum\limits_{r=1}^{m-k+1}\left(
(-1)^{m+1-r}q_{m+1,r}\prod\limits_{j=r}^{m}q_{j,j+1}\det
(Q_{k,r-1})\right)
\\
&&+\sum\limits_{r=m-k+2}^{m}\left(
(-1)^{m+1-r}q_{m+1,r}\prod\limits_{j=r}^{m}q_{j,j+1}\det
(Q_{k,r-1})\right)
\\
&=&t_{1}\det (Q_{k,m})+\sum\limits_{r=m-k+2}^{m}\left(
(-1)^{m+1-r}q_{m+1,r}\prod\limits_{j=r}^{m}q_{j,j+1}\det
(Q_{k,r-1})\right)
\\
&=&t_{1}\det (Q_{k,m})+\sum\limits_{r=m-k+2}^{m}\left( (-1)^{m+1-r}.i^{m+1-r}%
\frac{t_{m-r+2}}{t_{2}^{(m-r+1)}}\prod\limits_{j=r}^{m}it_{2}\det
(Q_{k,r-1})\right) \\
&=&t_{1}\det (Q_{k,m}) \\
&&+\sum\limits_{r=m-k+2}^{m}\left( (-1)^{m+1-r}.i^{m+1-r}\frac{t_{m-r+2}}{%
t_{2}^{(m-r+1)}}.i^{m+1-r}.t_{2}^{(m-r+1)}\det (Q_{k,r-1})\right) \\
&=&t_{1}\det (Q_{k,m})+\sum\limits_{r=m-k+2}^{m}\left(
(-1)^{m+1-r}.i^{m+1-r}t_{m-r+2}.i^{m+1-r}.\det (Q_{k,r-1})\right) \\
&=&t_{1}\det (Q_{k,m})+\sum\limits_{r=m-k+2}^{m}t_{m-r+2}\det (Q_{k,r-1}) \\
&=&t_{1}\det (Q_{k,m})+t_{2}\det (Q_{k,m-1})+\cdots +t_{k}\det
(Q_{k,m-(k-1)})
\end{eqnarray*}%
From the hypothesis and the definition of generalized Fibonacci
polynomials we
obtain%
\begin{equation*}
\det (Q_{k,m+1})=t_{1}F_{k,m+1}+t_{2}F_{k,m}+\cdots
+t_{k}F_{k,m-(k-2)}=F_{k,m+2}.
\end{equation*}%
Therefore, the result is true for all positive integers.
\end{proof}

\bigskip

\begin{exam}
We obtain $5$-th generalized Fibonacci polynomial for $k=6$, using Theorem
2.2%
\begin{equation*}
F_{6,5}(t)=\det \left[
\begin{array}{cccc}
t_{1} & it_{2} & 0 & 0 \\
i & t_{1} & it_{2} & 0 \\
\frac{-t_{3}}{t_{2}^{2}} & i & t_{1} & it_{2} \\
\frac{-it_{4}}{t_{2}^{3}} & \frac{-t_{3}}{t_{2}^{2}} & i & t_{1}%
\end{array}%
\right] =t_{4}+2t_{1}t_{3}+t_{2}^{2}+t_{1}^{4}+3t_{1}^{2}t_{2}.
\end{equation*}
\end{exam}

\bigskip

\begin{cor}
\bigskip $\left[ 10\right] $ Let $k\geq 2$ be an integer$,$ $f_{k,n}$ be the
generalized order-$k$ Fibonacci numbers (1) and $C_{k,n}=(c_{rs})$ $n\times
n $ Hessenberg matrix$,$ where%
\begin{equation*}
c_{rs}=\left\{
\begin{array}{l}
i^{\left\vert r-s\right\vert }\text{ \ \ \ \ \ \ \ \ \ \ \ \ \ \ \ \ \ \ if
\ }-1\leq r-s<k\text{ },\text{\ \ \ \ \ \ \ \ \ \ \ \ \ \ \ \ \ \ \ \ \ \ \
\ \ \ } \\
0\text{ \ \ \ \ \ \ \ \ \ \ \ \ \ \ \ \ \ \ \ \ \ \ otherwise\ \ \ \ \ \ \ \
\ \ \ \ \ \ \ \ \ \ \ \ \ \ \ \ \ \ \ \ \ \ \ \ \ \ }%
\end{array}%
\right.
\end{equation*}%
then%
\begin{equation*}
\det (C_{k,n})=f_{k,k+n-1}
\end{equation*}%
where $i=\sqrt{-1}.$
\end{cor}

\bigskip

\begin{proof}
\bigskip It is direct from Theorem 2.2 for $t_{i}=1.$
\end{proof}

\begin{thrm}
\bigskip Let $k\geq 2$ be an integer$,$ $F_{k,n}(t)$ be the generalized
Fionacci Polynomial (5) and $B_{k,n}=(b_{ij})$ be an $n\times n$ lower
Hessenberg matrix such that%
\begin{equation*}
b_{ij}=\left\{
\begin{array}{l}
-t_{2}\text{ \ \ \ \ \ \ \ \ \ \ \ if \ \ \ }j=i+1, \\
\frac{t_{i-j+1}}{t_{2}^{(i-j)}}\text{\ \ \ \ \ \ \ \ \ \ \ \ \ if\ \ \ \ \ }%
0\leq i-j<k\text{,} \\
0\text{\ \ \ \ \ \ \ \ \ \ \ \ \ \ otherwise}%
\end{array}%
\right.
\end{equation*}%
i.e.,%
\begin{equation}
B_{k,n}=\left[
\begin{array}{cccccc}
t_{1} & -t_{2} & 0 & 0 & \cdots  & 0 \\
1 & t_{1} & -t_{2} & 0 & \cdots  & 0 \\
\frac{t_{3}}{t_{2}^{2}} & 1 & t_{1} & -t_{2} & \cdots  & 0 \\
\vdots  & \vdots  & \vdots  & \vdots  &  & \vdots  \\
\frac{t_{k}}{t_{2}^{k-1}} & \frac{t_{k-1}}{t_{2}^{k-2}} & \frac{t_{k-2}}{%
t_{2}^{k-3}} & \frac{t_{k-3}}{t_{2}^{k-4}} & \cdots  & 0 \\
0 & \frac{t_{k}}{t_{2}^{k-1}} & \frac{t_{k-1}}{t_{2}^{k-2}} & \frac{t_{k-2}}{%
t_{2}^{k-3}} & \cdots  & 0 \\
& \vdots  & \vdots  & \vdots  & \ddots  &  \\
0 & 0 & 0 & \cdots  & \cdots  & t_{1}%
\end{array}%
\right]   \label{beka}
\end{equation}%
then%
\begin{equation*}
\det (B_{k,n})=F_{k,n+1}(t).
\end{equation*}%
where $t_{0}=1.$
\end{thrm}

\bigskip

\begin{proof}
\bigskip To prove $\det (B_{k,m})=F_{k,m+1}(t)$, we use the mathematical induction on $m$. The result is true for $%
m=1$ by hypothesis.

Using Theorem (2.1) we have
\begin{eqnarray*}
\det (B_{m+1,k}) &=&b_{m+1,m+1}\det
(B_{k,m})+\sum\limits_{r=1}^{m}((-1)^{m+1-r}b_{m+1,r}\prod%
\limits_{j=r}^{m}b_{j,j+1}\det (B_{r-1,k})) \\
&=&t_{1}\det
(B_{k,m})+\sum\limits_{r=1}^{m-k+1}((-1)^{m+1-r}b_{m+1,r}\prod%
\limits_{j=r}^{m}b_{j,j+1}\det (B_{r-1,k})) \\
&&+\sum\limits_{r=m-k+2}^{m}((-1)^{m+1-r}b_{m+1,r}\prod%
\limits_{j=r}^{m}b_{j,j+1}\det (B_{r-1,k})) \\
&=&t_{1}\det (B_{k,m})+\sum\limits_{r=m-k+2}^{m}((-1)^{m+1-r}.\frac{t_{m-r+2}%
}{t_{2}^{(m-r+1)}}\prod\limits_{j=r}^{m}(-t_{2})\det (B_{r-1,k})) \\
&=&t_{1}\det (B_{k,m}) \\
&&+\sum\limits_{r=m-k+2}^{m}((-1)^{m+1-r}.\frac{t_{m-r+2}}{t_{2}^{(m-r+1)}}%
.(-1)^{m+1-r}t_{2}^{(m-r+1)}\det (B_{r-1,k})) \\
&=&t_{1}\det (B_{k,m})+\sum\limits_{r=m-k+2}^{m}(t_{m-r+2}.\det
(B_{r-1,k}))
\\
&=&t_{1}\det (B_{k,m})+t_{2}\det (B_{k,m-1})+\cdots +t_{k}\det
(B_{k,m-(k-1)}).
\end{eqnarray*}%
From the hypothesis and the definition of generalized Fibonacci polynomials we obtain%
\begin{equation*}
\det (Q_{m+1,k})=t_{1}F_{k,m+1}+t_{2}F_{k,m}+\cdots
+t_{k}F_{k,m-(k-2)}=F_{k,m+2}.
\end{equation*}%
Therefore, the result is true for all positive integers.
\end{proof}

\bigskip

\begin{exam}
\bigskip We obtain $6$-th generalized Fibonacci polynomial for $k=4$, using
Theorem 2.5%
\begin{equation*}
F_{4,6}(t)=\det \left[
\begin{array}{ccccc}
t_{1} & -t_{2} & 0 & 0 & 0 \\
1 & t_{1} & -t_{2} & 0 & 0 \\
\frac{t_{3}}{t_{2}^{2}} & 1 & t_{1} & -t_{2} & 0 \\
\frac{t_{4}}{t_{2}^{3}} & \frac{t_{3}}{t_{2}^{2}} & 1 & t_{1} & -t_{2} \\
0 & \frac{t_{4}}{t_{2}^{3}} & \frac{t_{3}}{t_{2}^{2}} & 1 & t_{1}%
\end{array}%
\right] =2t_{1}t_{4}+2t_{2}t_{3}+t_{1}^{5}+3t_{1}t_{2}^{2}+3t_{1}^{2}t_{3}+%
\allowbreak 4t_{1}^{3}t_{2}.
\end{equation*}
\end{exam}

\bigskip

\begin{cor}
\bigskip $\left[ 10\right] $Let $k\geq 2$ be an integer$,$ $f_{k,n}$ be the
generalized order-$k$ Fibonacci numbers (1) and $M_{k,n}=(m_{ij})$ be an $%
n\times n$ lower Hessenberg matrix such that%
\begin{equation*}
m_{ij}=\left\{
\begin{array}{l}
-1\text{ \ \ \ \ \ \ \ \ \ \ \ if \ \ \ }j=i+1, \\
1\text{\ \ \ \ \ \ \ \ \ \ \ \ \ if\ \ \ \ \ }0\leq i-j<k\text{,} \\
0\text{\ \ \ \ \ \ \ \ \ \ \ \ \ \ otherwise}%
\end{array}%
\right.
\end{equation*}%
then%
\begin{equation*}
\det (M_{k,n})=f_{k,k+n-1}
\end{equation*}
\end{cor}

\bigskip

\begin{proof}
It is direct from Theorem 2.5 for $t_{i}=1.$
\end{proof}

\begin{cor}
\bigskip \bigskip\ If we rewrite Theorem (2.2) and Theorem (2.5) \newline
i) for $t_{i}=c_{j}$ $(1\leq i,j\leq k),$ we obtain%
\begin{equation*}
\det (Q_{k,n})=f_{k,n}^{1\text{ }}
\end{equation*}%
and%
\begin{equation*}
\det (B_{k,n})=f_{k,n}^{1\text{ }}
\end{equation*}%
respectively, \newline
ii) for $t_{1}=2$ and $t_{i}=1$ for $2\leq i\leq k,$ we obtain%
\begin{equation*}
\det (Q_{k,n})=p_{k,n+1}^{k\text{ }}
\end{equation*}%
and%
\begin{equation*}
\det (B_{k,n})=p_{k,n+1}^{k\text{ }}
\end{equation*}%
respectively, \newline
iii) for $t_{1}=0$ and $t_{i}=1$ for $2\leq i\leq k,$ we obtain%
\begin{equation*}
\det (Q_{k,n})=v_{k,n+1}^{k\text{ }}
\end{equation*}%
and%
\begin{equation*}
\det (B_{k,n})=v_{k,n+1}^{k\text{ }}
\end{equation*}%
respectively. Where $f_{k,n}^{1\text{ }}$, $p_{k,n}^{k\text{ }}$, $v_{k,n}^{k%
\text{ }}$ be the $k$ sequences of generalized order-$k$ Fibonacci, Pell and
Van der Laan numbers.  Matrices $Q_{k,n}$ and $B_{k,n}$ are as in (7) and
(8), respectively.
\end{cor}

\bigskip

\section{The permanent representations}

\bigskip

Let $A=(a_{i,j})$ be a square matrix of order $n$ over a ring R. The
permanent of $A$ is defined by%
\begin{equation*}
\text{per}(A)=\sum\limits_{\sigma \in
S_{n}}\prod\limits_{i=1}^{n}a_{i,\sigma (i)}
\end{equation*}%
where $S_{n}$ denotes the symmetric group on $n$ letters.

Let $A_{i,j}$ be the $(i,j)$-th minor of matrix $A.$ Then%
\begin{equation*}
\text{per}(A)=\sum\limits_{k=1}^{n}a_{i,k}\text{per}(A_{i,k})=\sum%
\limits_{k=1}^{n}a_{k,j}\text{per}(A_{k,j})
\end{equation*}%
for any $i,j$.

\bigskip

\begin{thrm}
$\left[ 10\right] $Let $A_{n}$ be $n\times n$ lower Hessenberg matrix for
all $n\geq 1$ and define per$(A_{0})=1.$ Then,%
\begin{equation*}
\text{per}(A_{1})=a_{11}
\end{equation*}%
and for $n\geq 2$%
\begin{equation}
\text{per}(A_{n})=a_{n,n}\text{per}(A_{n-1})+\sum\limits_{r=1}^{n-1}(a_{n,r}%
\prod\limits_{j=r}^{n-1}a_{j,j+1}\text{per}(A_{r-1})).
\end{equation}
\end{thrm}

\bigskip

\begin{thrm}
\bigskip \bigskip Let $k\geq 2$ be an integer, $F_{k,n}(t)$ be the
generalized Fibonacci Polynomial and $H_{k,n}=(h_{rs})$ be an
$n\times n$
lower Hessenberg matrix such that%
\begin{equation*}
h_{rs}=\left\{
\begin{array}{l}
i^{r-s}.\frac{t_{r-s+1}}{t_{2}^{(r-s)}}\text{ \ \ \ \ \ \ \ \ if \ }-1\leq
r-s<k\text{ \ \ \ \ \ \ \ \ \ \ \ \ \ \ \ \ \ \ \ \ \ \ \ \ \ \ \ \ \ \ } \\
0\text{ \ \ \ \ \ \ \ \ \ \ \ otherwise\ \ \ \ \ \ \ \ \ \ \ \ \ \ \ \ \ \ \
\ \ \ \ \ \ \ \ \ \ \ \ \ \ \ }%
\end{array}%
\right.
\end{equation*}%
i.e.,%
\begin{equation}
H_{k,n}=\left[
\begin{array}{cccccc}
t_{1} & -it_{2} & 0 & 0 & \cdots  & 0 \\
i & t_{1} & -it_{2} & 0 & \cdots  & 0 \\
i^{2}\frac{t_{3}}{t_{2}^{2}} & i & t_{1} & -it_{2} & \cdots  & 0 \\
\vdots  & \vdots  & \vdots  & \vdots  &  & \vdots  \\
i^{k-1}\frac{t_{k}}{t_{2}^{k-1}} & i^{k-2}\frac{t_{k-1}}{t_{2}^{k-2}} &
i^{k-3}\frac{t_{k-2}}{t_{2}^{k-3}} & i^{k-4}\frac{t_{k-3}}{t_{2}^{k-4}} &
\cdots  & 0 \\
0 & i^{k-1}\frac{t_{k}}{t_{2}^{k-1}} & i^{k-2}\frac{t_{k-1}}{t_{2}^{k-2}} &
i^{k-3}\frac{t_{k-2}}{t_{2}^{k-3}} & \cdots  & 0 \\
& \vdots  & \vdots  & \vdots  & \ddots  &  \\
0 & 0 & 0 & \cdots  & \cdots  & t_{1}%
\end{array}%
\right]
\end{equation}%
then%
\begin{equation*}
\text{per}(H_{k,n})=F_{k,n+1}(t)
\end{equation*}%
where $t_{0}=1$ and $i=\sqrt{-1}.$
\end{thrm}

\bigskip

\begin{proof}
\bigskip Since the proof is similar to the proof of Theorem (2.2) by using Theorem
(3.1) we omit the detail.
\end{proof}

\bigskip

\begin{exam}
\bigskip We obtain $7$-th generalized Fibonacci Polynomials for $k=5$, using
Theorem 3.2%
\begin{eqnarray*}
F_{5,7} &=&\text{per}\left[
\begin{array}{cccccc}
t_{1} & -it_{2} & 0 & 0 & 0 & 0 \\
i & t_{1} & -it_{2} & 0 & 0 & 0 \\
\frac{-t_{3}}{t_{2}^{2}} & i & t_{1} & -it_{2} & 0 & 0 \\
\frac{-it_{4}}{t_{2}^{3}} & \frac{-t_{3}}{t_{2}^{2}} & i & t_{1} & -it_{2} &
0 \\
\frac{t_{5}}{t_{2}^{4}} & \frac{-it_{4}}{t_{2}^{3}} & \frac{-t_{3}}{t_{2}^{2}%
} & i & t_{1} & -it_{2} \\
0 & \frac{t_{5}}{t_{2}^{4}} & \frac{-it_{4}}{t_{2}^{3}} & \frac{-t_{3}}{%
t_{2}^{2}} & i & t_{1}%
\end{array}%
\right] \\
&=&2t_{1}t_{5}+2t_{2}t_{4}+6t_{1}t_{2}t_{3}+t_{2}^{3}+t_{3}^{2}+t_{1}^{6}+%
\allowbreak
3t_{1}^{2}t_{4}+4t_{1}^{3}t_{3}+5t_{1}^{4}t_{2}+6t_{1}^{2}t_{2}^{2}.
\end{eqnarray*}
\end{exam}

\begin{cor}
\bigskip \bigskip \bigskip $\left[ 10\right] $Let $k\geq 2$ be an integer$,$
$f_{k,n}$ be the generalized order-$k$ Fibonacci numbers  and $%
H_{k,n}=(h_{rs})$ be an $n\times n$ lower Hessenberg matrix such that%
\begin{equation*}
h_{rs}=\left\{
\begin{array}{l}
i^{r-s}\text{\ \ \ \ \ \ \ if \ }-1\leq r-s<k\text{ \ \ \ \ \ \ \ \ \ \ \ \
\ \ \ \ \ \ \ \ \ \ \ \ \ \ \ \ \ \ } \\
0\text{ \ \ \ \ \ \ \ \ \ \ \ otherwise\ \ \ \ \ \ \ \ \ \ \ \ \ \ \ \ \ \ \
\ \ \ \ \ \ \ \ \ \ \ \ \ \ \ }%
\end{array}%
\right.
\end{equation*}%
then%
\begin{equation*}
\text{per}(H_{k,n})=f_{k,k+n-1}
\end{equation*}
\end{cor}

\bigskip

\begin{proof}
\bigskip It is direct from Theorem 3.2 for $t_{i}=1.$
\end{proof}

\begin{thrm}
Let $k\geq 2$ be an integer$,$ $F_{k,n}(t)$ be the generalized
Fibonacci Polynomial and $L_{k,n}=(l_{ij})$ be an $n\times n$ lower
Hessenberg matrix
such that%
\begin{equation*}
l_{ij}=\left\{
\begin{array}{l}
\frac{t_{i-j+1}}{t_{2}^{(i-j)}}\text{\ \ \ \ \ \ \ \ \ \ \ if \ }-1\leq i-j<k%
\text{,\ \ \ \ \ \ \ \ \ \ \ \ \ \ \ \ \ \ } \\
0\text{ \ \ \ \ \ \ \ \ \ \ \ \ \ otherwise\ \ \ \ \ \ \ \ \ \ \ \ \ \ \ \ \
\ \ \ \ \ \ \ \ \ \ \ \ \ \ \ \ \ }%
\end{array}%
\right.
\end{equation*}%
i.e.,%
\begin{equation}
L_{k,n}=\left[
\begin{array}{cccccc}
t_{1} & t_{2} & 0 & 0 & \cdots  & 0 \\
1 & t_{1} & t_{2} & 0 & \cdots  & 0 \\
\frac{t_{3}}{t_{2}^{2}} & 1 & t_{1} & t_{2} & \cdots  & 0 \\
\vdots  & \vdots  & \vdots  & \vdots  &  & \vdots  \\
\frac{t_{k}}{t_{2}^{k-1}} & \frac{t_{k-1}}{t_{2}^{k-2}} & \frac{t_{k-2}}{%
t_{2}^{k-3}} & \frac{t_{k-3}}{t_{2}^{k-4}} & \cdots  & 0 \\
0 & \frac{t_{k}}{t_{2}^{k-1}} & \frac{t_{k-1}}{t_{2}^{k-2}} & \frac{t_{k-2}}{%
t_{2}^{k-3}} & \cdots  & 0 \\
& \vdots  & \vdots  & \vdots  & \ddots  &  \\
0 & 0 & 0 & \cdots  & \cdots  & t_{1}%
\end{array}%
\right] .
\end{equation}%
where $t_{0}=1,$ then%
\begin{equation*}
\text{per}(L_{k,n})=F_{k,n+1}.
\end{equation*}
\end{thrm}

\bigskip

\begin{proof}
\bigskip This is similar to the proof of \ Theorem 2.5
using Theorem 3.1.
\end{proof}

\bigskip

\begin{cor}
$\left[ 9\right] $\bigskip Let $k\geq 2$ be an integer$,$ $f_{k,n}$ be the
generalized order-$k$ Fibonacci numbers  and $D_{k,n}=(d_{ij})$ be an $%
n\times n$ lower Hessenberg matrix such that%
\begin{equation*}
d_{ij}=\left\{
\begin{array}{l}
1\text{\ \ \ \ \ \ \ \ \ \ \ if \ }-1\leq i-j<k\text{,\ \ \ \ \ \ \ \ \ \ \
\ \ \ \ \ \ \ } \\
0\text{ \ \ \ \ \ \ \ \ \ \ \ \ \ otherwise\ \ \ \ \ \ \ \ \ \ \ \ \ \ \ \ \
\ \ \ \ \ \ \ \ \ \ \ \ \ \ \ \ \ }%
\end{array}%
\right.
\end{equation*}%
then%
\begin{equation*}
\text{per}(D_{k,n})=f_{k,k+n-1}.
\end{equation*}
\end{cor}

\begin{proof}
\bigskip It is direct from Theorem 3.5 for $t_{i}=1.$
\end{proof}

\begin{cor}
If we rewrite Theorem (3.2) and Theorem (3.5) \newline
i) for $t_{i}=c_{j}$ $(1\leq i, j\leq k), $ we obtain%
\begin{equation*}
\text{per}(H_{k,n})=f_{k,n}^{\text{ }1}
\end{equation*}%
and%
\begin{equation*}
\text{per}(L_{k,n})=f_{k,n}^{\text{ }1}
\end{equation*}%
respectively, \newline
ii) for $t_{1}=2$ and $t_{i}=1$ $(2\leq i\leq k),$ we obtain
\begin{equation*}
\text{per}(H_{k,n})=p_{k,n+1}^{\text{ }k}
\end{equation*}%
and%
\begin{equation*}
\text{per}(L_{k,n})=p_{k,n+1}^{\text{ }k}
\end{equation*}%
respectively, \newline
iii) for $t_{1}=0$ and $t_{i}=1$ $(2\leq i\leq k),$ we obtain
\begin{equation*}
\text{per}(H_{k,n})=v_{k,n+1}^{\text{ }k}
\end{equation*}%
and%
\begin{equation*}
\text{per}(L_{k,n})=v_{k,n+1}^{\text{ }k}
\end{equation*}%
respectively. Where $f_{k,n}^{k\text{ }}$, $p_{k,n}^{k\text{ }}$, $v_{k,n}^{k%
\text{ }}$ be the $k$ sequences of generalized order-$k$ Fibonacci,
Pell and Van der Laan numbers. Matrices $H_{k,n}$ and $L_{k,n}$  are
as in (10) and (11), respectively.
\end{cor}

\bigskip 

\bigskip

\end{document}